\documentclass{amsart}       
\usepackage{latexsym,amsmath,amssymb,graphicx}
\usepackage[all,web]{xy}
\usepackage[shortlabels]{enumitem}
\usepackage{color}
\usepackage{hyperref}
\usepackage{todonotes}

\usepackage{xcolor}

\def\xyellowspace{%
  \sbox0{\colorbox{yellow}{\strut\ }}
  \dimen0=\wd0\relax
  \hskip0pt\cleaders\box0\hskip\dimen0\hskip0pt}

\begingroup
\catcode`\ =\active%
\gdef\makeyellowspace{\let \xyellowspace\catcode`\ =\active}%
\endgroup

\def\?#1{\colorbox{yellow}{\strut#1}}

\DeclareFontFamily{OT1}{rsfs10}{}
\DeclareFontShape{OT1}{rsfs10}{m}{n}{ <-> rsfs10 }{}
\DeclareMathAlphabet{\mathscript}{OT1}{rsfs10}{m}{n}

\DeclareMathOperator{\Hom}{Hom}     
\DeclareMathOperator{\Tors}{Tors}    
\DeclareMathOperator{\Div}{Div}     
\DeclareMathOperator{\Pic}{Pic}     
\DeclareMathOperator{\Cl}{Cl}       
\DeclareMathOperator{\CaCl}{CaCl}   
\DeclareMathOperator{\rk}{rk}       
\DeclareMathOperator{\coker}{coker} 
\DeclareMathOperator{\mult}{mult}

\title[$\Pic(X)$ in $\Cl(X)$: the toric case]{Embedding the Picard group inside the class group: the case of $\Q$-factorial complete toric varieties}

\author{Michele Rossi         \and
        Lea Terracini 
}

\address{Dipartimento di Matematica, Universit\`a di Torino,
via Carlo Alberto 10, 10123 Torino} \email{michele.rossi@unito.it,
lea.terracini@unito.it}

\thanks{ The first author is partially supported by the I.N.D.A.M. as a member of the G.N.S.A.G.A.}


\theoremstyle{plain}
\newtheorem{theorem}{Theorem}[section]
\newtheorem{proposition}[theorem]{Proposition}

\newtheorem{thm-def}[theorem]{Theorem--Definition}
\newtheorem{corollary}[theorem]{Corollary}

\newtheorem{lemma}[theorem]{Lemma}

\newtheorem*{a-proposition}{Proposition}

\theoremstyle{remark}
\newtheorem{remark}[theorem]{Remark}

\newtheorem{example}[theorem]{Example}

\theoremstyle{definition}
\newtheorem{definition}[theorem]{Definition}

\newtheorem*{step I}{Step I}
\newtheorem*{step II}{Step II}
\newtheorem*{step III}{Step III}
\newtheorem*{step IV}{Step IV}

\newcommand{\halfline}{\vskip6pt}
\def \a{\alpha }
\def \b{\beta }

\def \s{\sigma }

\def \Ga{\Gamma }
\def \Si{\Sigma }

\def \u{\mathbf{u}}
\def \v{\mathbf{v}}
\def \n{\mathbf{n}}
\def \w{\mathbf{w}}
\def \t{\mathbf{t}}
\def \x{\mathbf{x}}
\def \z{\mathbf{z}}
\def \1{\mathbf{1}}
\def \0{\mathbf{0}}

\def\p2{\mathbb{P}^2}
\def\p3{\mathbb{P}^3}
\def\p4{\mathbb{P}^4}

\def\co{\mathcal{O}}

\def\rk{\operatorname{rk}}

\def\GL{\operatorname{GL}}

\def\Z{\mathbb{Z}}

\def\C{\mathbb{C}}
\def\R{\mathbb{R}}
\def\M{\mathbf{M}}
\def\Q{\mathbb{Q}}

\def\T{\mathbb{T}}

\def\SF{\mathcal{SF}}

\def\Ga{\Gamma}

\def\Weil{\mathcal{W}_\T}
\def\Cart{\mathcal{C}_\T}
\def\cM{\mathcal{M}}

\begin{document}

\begin{abstract}Let $X$ be a $\Q$-factorial complete toric variety over an algebraic closed field of characteristic $0$. There is a canonical injection of the Picard group ${\rm Pic}(X)$ in the group ${\rm Cl}(X)$ of classes of Weil divisors.  These two groups are finitely generated abelian groups; whilst the first one is a free group, the second one may have torsion. We investigate algebraic and geometrical conditions under which the image of ${\rm Pic}(X)$ in ${\rm Cl}(X)$ is contained in a free part of the latter group.\end{abstract}

\maketitle

\tableofcontents

\section*{Introduction}
Let $X$ be an irreducible and normal algebraic variety over an algebraic closed field $k$ of characteristic $0$. Then the group $H^0(X,\mathcal{K}^*/\co^*)$ of Cartier divisors of $X$ can be represented as the subgroup of locally principal divisors of the group $\Div(X)$ of Weil divisors \cite[Rem.~II.6.11.2]{Hartshorne}. Quotienting both this groups by their subgroup of principal divisors one realizes the group $\CaCl(X)$ of classes of Cartier divisors as a subgroup of the group $\Cl(X)$ of classes of Weil divisors. In addition it turns out a canonical isomorphism between $\CaCl(X)$ and the Picard group $\Pic(X)$ of classes of isomorphic line bundles on $X$ \cite[Propositions~II.6.13,15]{Hartshorne}, so giving a canonical injection
\begin{equation}\label{iniezione}
  \xymatrix{\Pic(X)\ \ar@{^(->}@<-1pt>[r]
  &\ \Cl(X)}\,.
\end{equation}
For this reason, in this paper we will not distinguish between linear equivalence classes of Cartier divisors and isomorphism classes of line bundles, so identifying $\CaCl(X)=\Pic(X)$.

Assume now that $\Cl(X)$ is finitely generated. It is well known that a finitely generated abelian group decomposes (non canonically) a a direct sum of a free part and its torsion subgroup. Both $\Pic(X)$ and $\Cl(X)$ may have non trivial torsion  and clearly (\ref{iniezione}) induces an injection $\Tors(\Pic(X))\hookrightarrow\Tors(\Cl(X))$. Then the following natural question arises:
\begin{itemize}
  \item[(*)] \emph{let $F_C$ and $F_W$ be \emph{free parts} of $\Pic(X)$ and $\Cl(X)$, respectively, that is free subgroups such that}
      \begin{equation*}
        \Pic(X)=F_C\oplus\Tors(\Pic(X))\,,\quad\Cl(X)=F_W\oplus\Tors(\Cl(X))\,;
    \end{equation*}
    \emph{under which conditions on $X$, $F_C\subseteq\Pic(X)$ and $F_W\subseteq\Cl(X)$ the injection (\ref{iniezione}) may induce an injection}
    \begin{equation}\label{iniezione-f}
      \xymatrix{F_C\ \ar@{^(->}@<-1pt>[r]
  &\ F_W}\,?
    \end{equation}
    \end{itemize}
    One should expect some geometric condition on $X$ guaranteeing the existence of the injection (\ref{iniezione-f}), but we could not find anything, in the current literature, helping us to answering problem (*). Motivated by algebraic considerations (see Proposition \ref{prop:equivalenze}) we call \emph{pure} a normal, irreducible algebraic variety $X$ such that $\Cl(X)$ is finitely generated and there exist free parts $F_C$ and $F_W$ positively answering problem (*) (see Definition~\ref{def:pure}). On the contrary if for each choice of free parts $F_C$ and $F_W$ the injection (\ref{iniezione}) does not induce any injection (\ref{iniezione-f}) then $X$ is called \emph{impure}. Obvious examples of pure varieties are given by those varieties $X$ whose class group $\Cl(X)$ is finitely generated and free, and by smooth varieties admitting a finitely generated class group. Examples of impure varieties are in general more involved: some of them are given in \S~\ref{ssez:esempi}.

    \halfline In the present paper we will consider the easier case of a $\Q$-factorial complete toric variety $X$, essentially for three reasons:
    \begin{enumerate}[(a)]
      \item $\Cl(X)$ is a finitely generated abelian group (see e.g. \cite[Thm.~4.1.3]{CLS})
      \item $\Pic(X)$ is free, i.e. $\Tors(\Pic(X))=0$ (see e.g. \cite[Prop.~4.2.5]{CLS})
      \item locally principal divisors can be easily described by means of principal divisors on affine open subsets of $X(\Si)$ associated with maximal cones of the fan $\Si$.
    \end{enumerate}
    Conditions (a) and (b) translate problem (*) in the following
    \begin{itemize}
      \item[(**)] \emph{under which conditions on $X$ there exists a free part $F$ of $\Cl(X)$ such that (\ref{iniezione}) induces an injection $\Pic(X)\hookrightarrow F$ ?}
    \end{itemize}
    The main result of the present paper is a sufficient condition for a $\Q$-factorial complete toric variety to be a pure variety. This is given by Theorem \ref{teo:gcdmult1} and can be geometrically summarized as follows:

    \begin{theorem}[see Theorem \ref{teo:gcdmult1} and Remark \ref{rem:geom}]\label{thm:intro}
      Let $X(\Si)$ be a $\Q$-factorial complete toric variety of dimension $n$. Then it admits a canonical covering $Y(\widehat{\Si})\twoheadrightarrow X(\Si)$, un\-ra\-mi\-fied in codimension $1$, such that the class group $\Cl(Y)$ is free (this follows by \cite[Thm.~2.2]{RT-QUOT}). Both $X$ and $Y$ are orbifolds \cite[Thm.~3.1.19\,(b)]{CLS}; let $$\{U_{\widehat{\s}}\}_{\widehat{\s}\in\widehat{\Si}(n)}$$
      be the collection of affine charts given by the maximal cones and covering $Y$. Calling $\mult(\widehat{\s})$ the maximum order of a quotient singularity in the affine chart $U_{\widehat{\s}}$ (i.e. the \emph{multiplicity} of the cone $\widehat{\s}\in\widehat{\Si}(n)$, see Definition~\ref{def:molteplicità}), $X$ is a pure variety if $m_\Si:=\gcd\{\mult(\widehat{\s})\,|\,\widehat{\s}\in\widehat{\Si}(n)\}$ is coprime with the order of the Galois group of the covering $Y\to X$.
    \end{theorem}

This is not a necessary condition: Example \ref{es:puro} gives a counterexample.

Section \S\,\ref{ssez:esempi} is entirely devoted to give non trivial examples of pure and impure varieties. A big class of non trivial examples of pure varieties is  exhibited in Section \S\,\ref{ssez:SF1}: namely it is given by
\begin{itemize}
\item \emph{all $\Q$-factorial complete toric varieties whose \emph{small $\Q$-factorial modifications (s$\Q$m)} are actually isomorphisms}.
\end{itemize}
Here \emph{s$\Q$m of $X$} means a birational map $f:X\dashrightarrow Y$ such that $f$ is an isomorphism in codimension 1 and $Y$ is still a complete $\Q$-factorial toric variety.
By the combinatorial point of view the previous geometric property translates in  requiring that \emph{there is a unique simplicial and complete fan $\Si$ admitting 1-skeleton given by $\Si(1)$}. Our proof that those varieties are pure (see Proposition~\ref{prop:unsolofan}) passes through showing that every maximal simplicial cone generated by rays in $\Si(1)$ and not containing any further element of $\Si(1)$ other than its generators (we call \emph{minimal} such a  maximal simplicial cone) is  actually a cone of a complete and simplicial fan. This fact produces a completion procedure of fans looking to be of some interest by itself (see Lemma\,\ref{lem:esistenzafan} and \S\,\ref{sssez:completamento}), when compared with standard  completion procedures  \cite[Thm.~III.2.8]{Ewald96}, \cite{EwaldIshida06}, \cite{Rohrer11}.

Further results of the present paper are given by:
\begin{itemize}
  \item[-] algebraic considerations given in \S~\ref{sez:algebra}; apart from the definition of a pure submodule given in Definition~\ref{def:puremodule} and some consequences appearing in Proposition~\ref{prop:equivalenze}, the rest of this section consists of original considerations, as far as we know;
  \item[-] a characterization of $\Pic(X)$ as a subgroup of $\Cl(X)$, when $X$ is a pure, $\Q$-factorial, complete, toric variety: in \cite[Thm.~2.9.2]{RT-LA&GD} we gave a similar characterization in the case of a poly weighted space (PWS: see Notation \ref{ssez:notazioni}), that is when $\Cl(X)$ is free; a first generalization was given in \cite[Thm.~3.2\,(3)]{Erratum} which is here improved in \S~\ref{sez:pic} and in particular by Theorem~\ref{thm:pic};
  \item[-] an  example of a 4-dimensional simplicial fan whose completions necessarily require the addition of some new rays (see Example \ref{ex:noncompletabile}): Ewald, in his book \cite{Ewald96}, already announced the existence of examples of this kind (see the Appendix to Chapter III in \cite{Ewald96}), but we were not able to recover it. We then believe that Example \ref{ex:noncompletabile} may fill up a lack in the literature on these topics.
\end{itemize}

This paper is structured as follows. The first section gives all needed algebraic ingredients. Section 2 is devoted to state and prove the main result, given by Theorem~\ref{teo:gcdmult1}, and produce examples of pure and impure varieties in the toric setup (see \S\,\ref{ssez:esempi} and \S\,\ref{ssez:SF1}). Section 3 gives the above mentioned characterization of $\Pic(X)$ as a subgroup of $\Cl(X)$ when $X$ is a pure, $\Q$-factorial, complete, toric variety.

\section{Algebraic considerations}\label{sez:algebra}
Let $A $ be a PID, $\cM$ be a finitely generated $A$-module of rank $r$, and $T=\Tors_A(\cM)$.
\begin{definition} A \emph{free part} of $\cM$ is a free submodule $L\subseteq  \cM$ such that $\cM=L\oplus T$.
\end{definition}
It is well known that a free part of $\cM$ always exists.\\
Let $H\subseteq \cM$ be a free submodule of rank $h$. In general it is not true that $H$ is contained in a free part of $\cM$. For example let $A=\Z$, $\cM=\Z\oplus \Z/2\Z$ and let $H=\langle(2,1)\rangle$.\\

\begin{proposition}\label{prop:sopra} There exist elements $a_1,\ldots, a_h\in A$ such that $a_1|a_2 |\ldots |a_h$ satisfying the following property: every free part $L$ of $\cM$ has a basis $\mathbf{f}_1,\ldots, \mathbf{f}_r$ such that $a_1\mathbf{f}_1+t_1, \ldots, a_h\mathbf{f}_h+t_h$ is a basis of $H$ for suitable $t_1,\ldots, t_h\in T$.
\end{proposition}
\begin{proof} Let $\pi:\cM\to \cM/T$ the quotient map. The group $\cM/T$ is free of rank $r$. The restriction of $\pi$ to $H$ is injective, since $H$ is free and $\ker(\pi)=T$. Therefore $\pi(H)$ is a subgroup of rank $h$ of the free group $\cM/T$. By the elementary divisor theorem, there exist a basis $\mathbf{\tilde f}_1,\ldots, \mathbf{\tilde f}_r $ and element $a_1,\ldots, a_h\in A$ such that $a_1|a_2 |\ldots |a_h$ and $a_1\mathbf{\tilde f}_1,\ldots, a_h\mathbf{\tilde f}_h$ is a basis of $\pi(H)$. Let $\mathbf{h}_1,\ldots, \mathbf{h}_h$ be the basis of $H$ such that $\pi(\mathbf{h}_i)=a_i\mathbf{\tilde f}_i$. \\
Now let $L$ be a free part of $\cM$. The decomposition $\cM=L\oplus T$ gives rise to a section $s:\cM/T \to L$, i.e. $\pi\circ s= id$. By putting $\mathbf{f}_i=s(\mathbf{\tilde f}_i)$ we get a basis $\mathbf{f}_1,\ldots, \mathbf{f}_r $ of $L$ such that $a_1\mathbf{f}_1,\ldots, a_h\mathbf{ f}_h$ is a basis of $s(\pi(H))$.  For $i=1,\ldots, h$ put $t_i=\mathbf{h}_i-a_i\mathbf{f}_i$. Then $t_i\in\ker(\pi)=T$, so that the claim is proved.
\end{proof}

\begin{remark} When $A=\Z$,
the objects whose existence is established by Proposition \ref{prop:sopra} are effectively computable. In fact, assume that $\cM= \Z^r\oplus T$, where $T$ is a finite group, and that $\mathbf{g}_1+s_1,\ldots, \mathbf{g}_h+s_h$ is a basis of $H$, with $\mathbf{g}_1,...,\mathbf{g}_h\in \Z^r$ and $s_1,\ldots, s_h\in T$. Let $G$ be the $h\times r$ matrix having rows $\mathbf{g}_1,...,\mathbf{g}_h$. It is possible to compute the Smith normal form $S$ of $G$ and matrices $U\in \GL_h(\Z)$, $V\in \GL_r(\Z)$ such that $UGV=S$. Then the rows of $V^{-1}$ give the basis $\mathbf{f}_1,\ldots, \mathbf{f}_r$, the diagonal entries of $S$ give $a_1,\ldots, a_h\in \Z$. Moreover we recover the elements $t_1,\ldots, t_h$ by putting (with the obvious notation)
$$\begin{pmatrix} t_1\\ \vdots  \\ t_h \end{pmatrix} = U \begin{pmatrix} s_1\\ \vdots  \\ s_h \end{pmatrix}.$$\end{remark}

The following definition is standard (see for example \cite[Ex. B-3.6]{Rotman})
\begin{definition}\label{def:puremodule}
Let $\cM$ be an $A$-module. A submodule $\cM'\subseteq \cM$ is said \emph{pure} if the following property is satisfied:
$$\hbox{if  $am\in \cM'$ for some $a\in A,m\in \cM$, then there is $m'\in \cM'$ such that $am'=am$. }$$
\end{definition}

\begin{proposition}\label{prop:equivalenze}

The following are equivalent:
\begin{itemize}
\item[a)] $H$ is contained in a free part of $\cM$.
\item[b)] The image of $T$ in $\cM/H$ is a free summand.
\item[c)] The image of $T$ in $\cM/H$ is a pure submodule.
\item[d)] Let  $L$ be a free part of $\cM$ and  $\mathbf{f}_1,\ldots, \mathbf{f}_r$ be a basis of $L$ as in Proposition \ref{prop:sopra}; then for $i=1,\ldots, h$ the element $t_i$ is divisible by $a_i$ in $T$, that is there exists $u_i\in T$ such that $t_i=a_iu_i$;
\end{itemize}
\end{proposition}
\begin{proof} a) $\Rightarrow$ b): let $L$ be a free part of $\cM$ such that $H\subseteq L$. Then $\cM/H = (L\oplus T)/H \cong (L/H)\oplus T$.\\
The equivalence of $(b)$ and $(c)$ is a  the well-known fact that, for modules finitely generated over a PID, pure submodules and direct summands coincide (see for example \cite[Ex. B-3.7 (ii)]{Rotman}).\\
c) $\Rightarrow$ d): since $a_i\mathbf{f}_i+t_i\in H$, the image of $a_i\mathbf{f}_i$ in $\cM/H$ belongs to the image of $T$, for $i=1,\ldots, r$. By purity, there exists $u_i\in T$ such that the images of $a_i\mathbf{f}_i$ and $-a_iu_i$ coincide in $\cM/H$, that is $a_i\mathbf{f}_i +a_iu_i$ in $H$. But then $t_i-a_iu_i\in H\cap T=\{0\}$, because $H$ is free. \\
c) $\Rightarrow$ d): let $L'$ be the submodule of $\cM$ generated by $\mathbf{f}_1+u_1,\ldots,\mathbf{f}_h+u_h, \mathbf{f}_{h+1},\ldots, \mathbf{f}_r.$ Then $L'$ is a free part of $\cM$ containing $H$.

\end{proof}

Notice that since $H$ is free, $H\cap T=\{0\}$, so that the image of $T$ in $\cM/H$ is isomorphic to $T$.

For every prime element $p$ of $A$, we denote by $A_{(p)}$ the localization of $A$ at the prime ideal $(p)$. If $\cM$ is an $A$-module, $\cM_{(p)}$ is the localized $A_{(p)}$-module. \\
The localization of $T$ at $(p)$ coincide with the $p$-torsion of $T$, and $T=\bigoplus_p T_{(p)}$.\\
If $L$ is a free part of $\cM$ and $\cM=L\oplus T$ is the corresponding decomposition then $L_{(p)}$ is a free part of $\cM_{(p)}$, that is there is a decomposition $\cM_{(p)}=L_{(p)}\oplus T_{(p)}$. The natural map $\cM\to \cM_{(p)}$ is the sum of the injection $L\to L_{(p)}$ and the surjection $T\to T_{(p)}$.

\begin{proposition} \label{prop:localizza} $H$ is contained in a free part of $\cM$ if and only if
$H_{(p)}$ is contained in a free part of $\cM_{(p)}$  for every prime element $p\in A$.
\end{proposition}
\begin{proof} If $H$ is contained in a free part $L$ of $\cM$ then $L_{(p)}$ is a free part of $\cM_{(p)}$ containing $H_{(p)}$, every prime $p$ of $A$.  Conversely suppose that $H_{(p)}$ is contained in a free part of $\cM_{(p)}$  for every prime element $p\in A$. Let $L$ be a free part of $\cM$,  and $a_1,\ldots, a_h\in A$,  $\mathbf{f}_1,\ldots, \mathbf{f}_r\in L$, $t_1,\ldots, t_h\in T$ as in Proposition \ref{prop:sopra}.  Then $L_{(p)}$ is a free part of $\cM_{(p)}$ having basis $\mathbf{f}_1,\ldots, \mathbf{f}_r$  and  $a_1\mathbf{f}_1+\tilde t_1, \ldots, a_h\mathbf{f}_h+\tilde t_h$ is a basis of $H_{(p)}$ where $\tilde t_i\in T_{(p)}$ is the image of $t_i$ by the map $T\to T_{(p)}$.  Therefore $\tilde t_i$ is the component of $t_i$ in the $p$-torsion of $T$.  By Proposition \ref{prop:equivalenze} b), $\tilde t_i$ is divisible by $a_i$ in $ T_{(p)}$. Since this holds for every $p$, we obtain that $t_i$ is divisible by $a_i$ in $T$. Applying again Proposition \ref{prop:equivalenze} we see that $H$ is contained in a free part of $\cM$.\end{proof}

\section{Application to toric varieties}

As already mentioned in the Introduction, we put the following

\begin{definition}\label{def:pure}
Let $X$ be an irreducible and normal algebraic variety such that $\Cl(X)$ is finitely generated. Then $X$ is called \emph{pure} if there exist free parts $F_C$ and $F_W$ of $\Pic(X)$ and $\Cl(X)$, respectively, such that the canonical injection $\Pic(X)\hookrightarrow\Cl(X)$ descends to give the following commutative diagram
\begin{equation*}
  \xymatrix{\Pic(X)\ \ar@{^(->}@<-1pt>[r] &\ \Cl(X)\\
            F_C\ \ar@{^(->}@<-1pt>[r]\ar@{^(->}[u] &\ F_W\ar@{^(->}[u]}
\end{equation*}
In particular, if $X=X(\Si)$ is a $n$-dimensional toric variety whose $n$-skeleton $\Si(n)$ is not empty (see the following \S~\ref{ssez:notazioni} for notation) then $\Pic(X)$ is free (see e.g. \cite[Prop.~4.2.5]{CLS}), meaning that $X$ is pure if and only if $\Pic(X)$
is contained in a free part of $\Cl(X)$.\\
If $X$ is not pure, it  is called \emph{impure}.
\end{definition}

Of course, if $\Cl(X)$ is free then $X$ is pure; morover if $X$ is smooth and $\Cl(X)$ is finitely generated, then it is pure, beacause the injection $\Pic(X)\hookrightarrow \Cl(X)$ is an isomorphism.\\
 Conversely, producing examples of impure varieties is definitely more complicated (see Example~\ref{es:impuro}).

\subsection{Notation on toric varieties}\label{ssez:notazioni}

Let $X=X(\Si)$ be a $n$-dimensional toric variety associated with a fan $\Si$. Calling $\T\cong (k^*)^n$ the torus acting on $X$, we use the standard notation $M$, for the characters group of $\T$, and $N:=\Hom(M,\Z)$. Then $\Si$ is a collection of cones in $N_\R:=N\otimes\R\cong\R^n$. $\Si(i)$ denotes the $i$-skeleton of $\Si$, that is the collection of $i$-dimensional cones in the fan $\Si$. We shall use the notation $\tau\preceq \sigma$ to indicate that the  cone $\tau$ is a face of $\sigma$.

Given a toric variety $X(\Si)$ we will denote by $\Weil(X)\subseteq\Div(X)$ the subgroup of \emph{torus invariant} Weil divisors and by $\Cart(X)\subseteq\Weil(X)$ the subgroup of Cartier torus invariant divisors. It is well known that
\begin{equation*}
  \Weil(X)=\bigoplus_{\rho\in\Si(1)}\Z\cdot D_\rho\quad\text{where}\quad D_\rho:=\overline{\T\cdot x_\rho}
\end{equation*}
the latter being the closure of the torus orbit of the \emph{distinguished point} $x_\rho$ of the ray $\rho$ \cite[\S~3.2, \S~4.1]{CLS}. In particular the homomorphism $D\mapsto [D]$, sending a Weil divisor to its linear equivalence class, when restricted to torus invariant divisors still gives a epimorphism $d_X:\Weil(X)\twoheadrightarrow\Cl(X)$ \cite[Thm.~4.1.3]{CLS}.

In \cite[Def.~2.7]{RT-LA&GD} we introduced the notion of a \emph{poly weighted space} (PWS), which is a $\Q$-factorial complete toric variety $Y$ whose class group $\Cl(Y)$ is free. This is equivalent to say that $Y$ is connected in codimension 1 (1-connected), that is the regular locus $Y_\text{reg}$ of $Y$ is simply connected, as $Y$ is a normal variety: recall that $\pi_1(Y_\text{reg})\cong \Tors(\Cl(Y))=0$ \cite[Cor.~1.8, Thm.~2.1]{RT-QUOT}.
As proved in \cite[Thm.~2.2]{RT-QUOT}, every $\Q$-factorial complete  toric variety $X(\Si)$ is a finite quotient of a unique PWS $Y(\widehat{\Si})$, which is its universal \emph{covering unramified in codimension 1} ($1$-covering). The Galois group of the torus equivariant covering $Y\twoheadrightarrow X$ is precisely the dual group $\mu(X)=\Hom(\Tors(\Cl(X)),k^*)$. At lattice level, the equivariant surjection $Y\twoheadrightarrow X$ induces an injective automorphism $\beta: N\hookrightarrow N$ whose $\R$-linear extension $\beta_\R:N_\R\hookrightarrow N_\R$ identifies the associated fans, that is $\beta_\R(\widehat{\Si})=\Si$. Recall that one has the following commutative diagram (see diagram (5) in \cite{Erratum})
  \begin{equation}\label{div-diagram-covering}
    \begin{array}{c}
      \xymatrix{&&&0\ar[d]&\\
& 0 \ar[d] & 0 \ar[d] & \ker(\overline{\a})=T \ar[d] & \\
0 \ar[r] & M \ar[r]^-{div_X}_-{V^T}\ar[d]_-{\b^T} &
\Weil (X)=\Z^{|\Si(1)|} \ar[r]^-{d_X}_-{Q\oplus \Gamma}\ar[d]^-{\a}_-{\mathbf{I}_{n+r}} & \Cl(X)\cong \Z^r\oplus T \ar[r]\ar[d]^-{\overline{\a}} & 0 \\
0 \ar[r] & M \ar[r]^-{div_Y}_-{\widehat{V}^T}\ar[d]&\Weil(Y)=\Z^{|\widehat{\Si}(1)|}\ar[r]^-{d_Y}_-{Q}\ar[d] & \Cl (Y)\cong \Z^r \ar[r]\ar[d] & 0 \\
 & \coker(\b^T)\cong T\ar[d] & 0 & 0 & \\
 &0&&&}
    \end{array}
\end{equation}
where
\begin{itemize}
  \item $T=\Tors(\Cl(X))$\,;
  \item $div_X,div_Y$ are the morphisms sending a character in $M$ to the associated principal divisor in $\Weil(X),\Weil(Y)$, respectively;
  \item $d_X,d_Y$ are the morphisms sending a torus invariant divisor in $\Weil(X),\Weil(Y)$, respectively, to its class in $\Cl(X),\Cl(Y)$, respectively;
  \item $\a$ is the identification $\Weil(X)\cong\Weil(Y)$ induced by inverse images of rays by $\b_\R$, that is
      $$\a\left(\sum_{\rho\in\Si(1)} a_\rho D_{\rho}\right)= \sum_{\b_\R^{-1}(\rho)\in\widehat{\Si}(1)} a_\rho D_{\b_\R^{-1}(\rho)}\,;$$
  \item $\overline{\a}$ is what induced by $\a$ on classes groups;
  \item $V,\widehat{V}$ are matrices whose transposed represent $div_X,div_Y$, respectively, w.r.t. a chosen a basis of $M$ and standard bases of torus orbits of rays of $\Weil(X)$ and $\Weil(Y)$, respectively; since $|\Si(1)|=|\widehat{\Si}(1)|=n+r$, where $$r=\rk(\Cl(Y))=\rk(\Cl(X))$$
      both $V$ and $\widehat{V}$ are $n\times(n+r)$ integer matrices called \emph{fan matrices} of $X$ and $Y$, respectively; they turns out to be \emph{$F$-matrices}, in the sense of \cite[Def.~3.10]{RT-LA&GD}, and $\widehat{V}$ is also a \emph{$CF$-matrix}; notice that, still calling $\b$ the representative matrix of the homonymous morphism $\b:N\hookrightarrow N$ w.r.t. the basis dual to that chosen in $M$, there is the relation
      \begin{equation*}
        V=\b\cdot\widehat{V}
      \end{equation*}
      (see \cite[Prop.~3.1\,(3)]{RT-LA&GD} and \cite[Rem.~2.4]{RT-QUOT}); concretely, both $V$ and $\widehat{V}$ can be obtained as matrices whose columns represent primitive generators of rays in $\Si(1)$ and $\widehat{\Si}(1)$, respectively, w.r.t. the dual basis i.e.
      \begin{equation*}
        V=\left(
            \begin{array}{ccc}
              \v_1 & \cdots & \v_{n+r} \\
            \end{array}
          \right)\ ,\quad
          \widehat{V}=\left(
            \begin{array}{ccc}
              \widehat{\v}_1 & \cdots & \widehat{\v}_{n+r} \\
            \end{array}
          \right)
      \end{equation*}
      where $\Si(1)=\{\langle\v_1\rangle, \ldots ,\langle\v_{n+r}\rangle\}$, $\widehat{\Si}(1)=\{\langle\widehat{\v}_1\rangle, \ldots ,
      \langle\widehat{\v}_{n+r}\rangle\}$, being $\langle \v\rangle$ the ray generated by $\v$ in $\R^n\cong N_\R$\,;
  \item $Q$ is a matrix representing $d_Y$ w.r.t. a chosen basis of $\Cl(Y)$; it is a $r\times(n+r)$ integer matrix which turns out to be a \emph{Gale dual matrix} of both $V$ and $\widehat{V}$, in the sense of \cite[\S~3.1]{RT-LA&GD} and a \emph{$W$-matrix}, in the sense of \cite[Def.~3.9]{RT-LA&GD}; it is called a \emph{weight matrix} of both $X$ and $Y$;
  \item the choice of a basis of $\Cl(Y)$ as above, determines a basis of a free part of $\Cl(X)$; complete such a basis with a set of generators of the torsion subgroup $T\subseteq\Cl(X)$; then $d_X$ decomposes as $d_X=f_X\oplus\tau_X$ where
      \begin{equation*}
        \xymatrix{&\Z^r\\
        \Weil(X)\ar@{->>}[r]^-{d_X}\ar@{->>}[ur]^-{f_X}\ar@{->>}[dr]_-{\tau_X}&\Cl(X)\cong\Z^r\oplus T
        \ar@{->>}[u]^-{\pi_1}\ar@{->>}[d]_-{\pi_2}\hskip1.4cm ;\\
        &T}
      \end{equation*}
      with respect to these choices, the weight matrix $Q$ turns out to be a representative matrix of $f_X$, too, while morphism $\tau_X$ is represented by a \emph{torsion matrix} $\Ga$ \cite[Thm.~3.2\,(6)]{Erratum}.
\end{itemize}

\subsubsection{Some further notation}\label{sssez:notazioni}
Let $A\in\mathbf{M}(d,m;\Z)$ be a $d\times m$ integer matrix, then
\begin{eqnarray*}
  &\mathcal{L}_r(A)\subseteq\Z^m& \text{denotes the sublattice spanned by the rows of $A$;} \\
  &\mathcal{L}_c(A)\subseteq\Z^d& \text{denotes the sublattice spanned by the columns of $A$;} \\
  &A_I\,,\,A^I& \text{$\forall\,I\subseteq\{1,\ldots,m\}$ the former is the submatrix of $A$ given by}\\
  && \text{the columns indexed by $I$ and the latter is the submatrix of}\\
  && \text{$A$ whose columns are indexed by the complementary }\\
  && \text{subset $\{1,\ldots,m\}\backslash I$;}
\end{eqnarray*}
Given a fan matrix $V=(\v_1,\ldots,\v_{n+r})\in\mathbf{M}(n,n+r;\Z)$ then
\begin{eqnarray*}
  &\langle V\rangle=\langle\v_1,\ldots,\v_{n+r}\rangle\subseteq N_{\R}& \text{denotes the cone generated by the columns of $V$;} \\
  &\SF(V)=\SF(\v_1,\ldots,\v_{n+r})& \text{is the set of all rational simplicial and complete}\\
  && \text{ fans $\Si$ such that $\Sigma(1)=\{\langle\v_1\rangle,\ldots,\langle\v_{n+r}\rangle\}\subset N_{\R}$}\\  && \text{(see \cite[Def.~1.3]{RT-LA&GD}).}
\end{eqnarray*}
Given a  fan $\Si\in \SF(V) $ we put
\begin{eqnarray*}\label{ISigma}
    \mathcal{I}^{\Si}&=&\{I\subseteq\{1,\ldots,n+r\}:\left\langle V^I\right\rangle\in\Si(n)\}.
   \end{eqnarray*}

\subsection{A sufficient condition}

This section is aimed to give a sufficient condition for a $\Q$-factorial complete toric variety to be a pure variety. Let us, first of all, outline some equivalent facts.

\begin{proposition}\label{prop:localizza2}
Let $X(\Si)$ be a $\Q$-factorial complete toric variety, $Y(\widehat{\Si})\twoheadrightarrow X(\Si)$ be its universal 1-covering, $V$ and $\widehat{V}$ be fan matrices of $X$ and $Y$, respectively. Assuming notation as in diagram (\ref{div-diagram-covering}), the following are equivalent:
\begin{itemize}
\item[a)] $X$ is a pure variety;
\item[b)] there is a decomposition $\Weil(X)=\mathcal{L}_r(\widehat V)\oplus F$ such that $\Cart(X)\subseteq \mathcal{L}_r(V)\oplus F$;
\item[c)] for every prime $p$ there exists a $\Z_{(p)}$-module $F_p$ and a decomposition  $$\Weil(X)_{(p)}=\mathcal{L}_r(\widehat{V})_{(p)}\oplus F_p$$
    such that $\Cart(X)_{(p)}\subseteq \mathcal{L}_r(V)_{(p)}\oplus F_p$.
\end{itemize}
\end{proposition}

\begin{proof} a) $\Rightarrow$ b): if $X$ is a pure variety, let $\Cl(X)=L\oplus T$ be a decomposition such that $L$ is a free part and $\Pic(X)\subseteq L$. We can identify $L$ with $\Z^r$ in the first row of diagram (\ref{div-diagram-covering}). Let $s:L\to  \Weil(X)$ be any section (i.e. $Q\circ s=id_L$ ) and put $F=s(L)$. If $x\in \Weil(X)$, write $d_X(x)=a+b$, with $a\in L$ and $b\in T$. Then $Q(x-s(a))=0$ so that $x-s(a)\in\mathcal{L}_r(\widehat V)$; this proves that $\Weil(X)=\mathcal{L}_r(\widehat V)\oplus F$. If $x\in \Cart(X)$ then write $x=a+b$ with $a\in\mathcal{L}_r(\widehat V)$ and $b\in F$; since $d_X(x)\in L$, we have $\Ga\cdot x=\Ga\cdot a=0$, so that $a\in\mathcal{L}_r(V)$.\\
b) $\Rightarrow$ c) is obvious.\\
c) $\Rightarrow $ a): let $p$ be a prime and put $F'_p=F_p\cap \Cart(X)_{(p)}$. We have $$\Cart(X)_{(p)}=\mathcal{L}_r(V)_{(p)}\oplus F'_p$$
so that
\begin{eqnarray*}
  \Cl(X)_{(p)}/\Pic(X)_{(p)}=\Weil(X)_{(p)}/\Cart(X)_{(p)} &\cong& \left(\mathcal{L}_r(\widehat V)_{(p)}/\mathcal{L}_r(V)_{(p)}\right)\oplus \left(F_p/F'_p\right) \\
   &\cong& T_{(p)}\oplus \left(F_p/F'_p\right)
\end{eqnarray*}
Then we see that the image of $T_{(p)}$ is a direct summand in $\Cl(X)_{(p)}/\Pic(X)_{(p)}$, so that $\Pic(X)_{(p)}$ is contained in a free part of $\Cl(X)_{(p)}$ by Proposition \ref{prop:equivalenze} b). Since this holds for every $p$ we can apply Proposition \ref{prop:localizza} and deduce that $\Pic(X)$ is contained in a free part of $\Cl(X)$, so that $X$ is pure.
\end{proof}

\begin{definition}\label{def:molteplicità}
Let $\Sigma$ be a fan in $\R^n$. For every simplicial cone $\sigma\in \Sigma$, let $\mathbf{w}_1,\ldots, \mathbf{w}_k\in \Z^n$ be the set of minimal generators of $\sigma$. Let $\mathcal{V}$ be the subspace of $\R^n$ generated by $\sigma$, and $L=\mathcal{V}\cap \Z^n$.  The \emph{multiplicity} of $\sigma$ is the index
$$\mult(\sigma)=[L: \Z\mathbf{w}_1\oplus\ldots\oplus  \Z\mathbf{w}_k].$$
If $\Sigma$ is a simplicial fan we put
$$m_\Sigma=\gcd\{\mult(\sigma)\ |\ \sigma \hbox{ is a maximal cone in } \Sigma \}.$$
\end{definition}

Set, once for all, the following notation:
\begin{equation}\label{EI}
\forall\,I\subseteq\{1,\ldots,n+r\}\quad
E_I:=\{\x=(x_1,\ldots, x_{n+r})\in\Z^{n+r}\ |\ x_i=0, \forall i\not\in I\}\,.
\end{equation}

We are now in a position to state and prove the main result of the present paper.

\begin{theorem}
\label{teo:gcdmult1}
Let $X=X(\Sigma)$ be a complete $\Q$-factorial toric variety and $Y=Y(\widehat{\Sigma})$ be its universal 1-covering; let $\widehat{V}$ be a fan matrix associated to $Y$, and $V=\beta\cdot\widehat{V}$ be a fan matrix associated to $X$. Suppose that $(\det(\beta), m_{\widehat{\Sigma}})=1$. Then $X$ is a pure variety.
\end{theorem}
\begin{proof}
 By Proposition \ref{prop:localizza2}, it suffices to show that for every prime $p$ there exists a $\Z_{(p)}$-module $F_p$ and a decomposition  $\Weil(X)_{(p)}=\mathcal{L}_r(\widehat{V})_{(p)}\oplus F_p$ such that $\Cart(X)_{(p)}\subseteq \mathcal{L}_r(V)_{(p)}\oplus F_p$. If $p\not | \det(\beta)$ then $\mathcal{L}_r(V)_{(p)}=\mathcal{L}_r(\widehat{V})_{(p)}$ and we are done. Assume that $p|\det(\beta)$;  by hypothesis there exists a maximal cone $\widehat{\sigma}=\widehat{\sigma}^I\in \widehat{\Sigma}$ such that $p\not\hskip-.045cm |\mult(\widehat{\sigma})=\det(Q_I)$. Put
$F_p=E_{I,{(p)}}$, where $E_I$ is defined in (\ref{EI}). By definition $\Cart(X)\subseteq \mathcal{L}_r(V)\oplus E_I$, so that $\Cart(X)_{(p)}\subseteq \mathcal{L}_r(V)_{(p)}\oplus F_p$. We claim that $\Z_{(p)}^{n+r}=\mathcal{L}_r(\widehat{V})_{(p)}\oplus F_p$. The inclusion $\supseteq$ being obvious, assume that $\x\in\Z_{(p)}^{n+r}$. Since $\det(Q_I)$ is invertible in $\Z_{(p)}$, there esists $\mathbf{y}\in E_I$ such that $Q\x=Q\mathbf{y}$, that is $\x-\mathbf{y}\in\ker(Q)=\mathcal{L}_r(\widehat{V})$.
\end{proof}

\begin{corollary}\label{cor:gcdmult}
Let $Y=Y(\widehat{\Sigma})$ be a poly weighted projective space such that $m_{\widehat{\Sigma}}=1.$  Then every $\Q$-factorial complete toric variety having $Y$ as universal $1$-covering is pure.
\end{corollary}

\begin{remark}\label{rem:geom}
  Geometrically the previous Theorem~\ref{teo:gcdmult1} translates precisely in Theorem~\ref{thm:intro} stated in the introduction. In fact a $\Q$-factorial complete toric variety is an orbifold (see \cite[Thm.~3.1.19\,(b)]{CLS}) whose $n$-skeleton parameterizes a covering by affine charts. In particular $Y$ has only finite quotient singularities whose order is necessarily a divisor of some multiplicity $\mult(\widehat{\s})$, for $\widehat{\s}\in\widehat{\Si}(n)$. Moreover, the affine chart $U_{\widehat{\s}}$ has always a quotient singularity of maximum order $\mult(\widehat{\s})$. Hence Theorem~\ref{thm:intro} follows.\\
  In particular the previous Corollary \ref{cor:gcdmult} gives the following
\end{remark}

\begin{corollary}
  Let $Y$ be a $n$-dimensional, $\Q$-factorial, complete toric variety admitting a torus invariant, Zarisky open subset $U\subseteq Y$, biregular to $\C^n$. Then $Y$ is a PWS and every $\Q$-factorial complete toric variety having $Y$ as universal 1-covering is pure.
\end{corollary}

\subsection{Examples}\label{ssez:esempi}
The present section is devoted to give some examples of pure and impure $\Q$-factorial complete toric varieties.

\begin{example}\label{es:impuro}
Consider the fan matrix
\begin{equation*}\widehat{V}=\begin{pmatrix}  1 &-1& 2& -3& -1\\ 1& -1& -1& 2& -1\\ 1& 1& 1& 1& -5\end{pmatrix}\end{equation*}

The corresponding weight matrix is
$$Q=\begin{pmatrix} 3&1& 10& 6& 4\\ 3& 2& 0& 0& 1\end{pmatrix}.$$
One can check that $|\SF(\widehat{V})|=2$. These two fans are given by taking all the faces of the following lists of maximal cones:
\begin{eqnarray*}
\widehat{\Sigma}_1&=&\{\langle 1, 2, 3\rangle ,\langle 1, 2, 4\rangle, \langle 2, 4, 5\rangle , \langle 1, 4, 5\rangle  , \langle 2, 3, 5\rangle , \langle 1, 3, 5\rangle  \}\\
\widehat{\Sigma}_2&=&\{ \langle 1, 3, 4\rangle , \langle 2, 3, 4\rangle ,\langle 2, 4, 5\rangle , \langle 1, 4, 5\rangle , \langle 2, 3, 5\rangle , \langle 1, 3, 5\rangle \}
\end{eqnarray*}
(We denote by $\langle i,j,k\rangle $ the cone generated by the columns $\v_i,\v_j,\v_k$ of the matrix $\widehat{V}$).
The list of multiplicities of maximal cones for the two fans are, respectively,
$$6,10 ,30, 20,18,12 \hbox{ and } 7,9, 30, 20,18,12,$$
so that $$m_{\widehat{\Sigma}_1}=2,\quad\quad m_{\widehat{\Sigma}_2}=1.$$
Define
$$\beta:=\begin{pmatrix}  1 &0& 0\\ 0& 1& 0 \\ 0& 0& 2\end{pmatrix} $$

and
$$V:=\beta\cdot\widehat{V}=\begin{pmatrix}  1 &-1& 2& -3& -1\\ 1& -1& -1& 2& -1\\ 2& 2& 2& 2& -10\end{pmatrix}.$$
A torsion matrix $\Gamma$ with entries in $\Z/2\Z$ such that $Q\oplus\Gamma$ represents the morphism assigning to each divisor its class, as in the previous diagram (\ref{div-diagram-covering}), is given by
$$\Gamma=\left(
           \begin{array}{ccccc}
             0 & 1 & 1 & 1 & 0 \\
           \end{array}
         \right)\,.$$
Let $\Sigma_1$ be the fan in $\SF(V)$ corresponding to $\widehat{\Sigma}_1$. We show that $X(\Sigma_1)$ is an impure variety.
Using methods explained in \cite[Thm.~3.2\,(2)]{Erratum}, we obtain that a basis of $\Cart(X)$ is given by the rows of the following matrix
$$C_X=\begin{pmatrix} 40 &0 &0&0&0\\
0 &60&0&0&0\\
0&0&3&0&0\\
-24 &-24&0&1&0\\
-9 &-47&-2&0&1
\end{pmatrix}.$$
Then
$$
Q\cdot C_X^T =\begin{pmatrix}
120 &60&30&-90&-90\\
120 &120&0&-120&-120
\end{pmatrix},\quad\quad
\Gamma\cdot C_X^T = \begin{pmatrix}
0&0&1&1&1
\end{pmatrix}
$$
Then we see that $\Pic(X)$ is generated in $\Cl(X)\cong  \Z^2\oplus\Z/2\Z$  by elements
$$ (120,120), (60,120), (30,0)+[1]_2,(90,120)+[1]_2\,.$$
 the first and the last of them are obviously generated by the remaining two elements, so that  $\Pic(X)$ is generated by $(60,120)$ and $(30,0)+[1]_2$. Every  free part of $\Cl(X)$  contains an element $z $ of the form $(15,0)+[a]_2$ for some $a\in\{0,1\}$;  therefore it must contain $2z=(30,0)$; then $ (30,0)+[1]_2$ cannot belong to any free part, meaning that $X(\Si_1)$ is impure.

Notice that purity is a property depending on the fan choice. In fact $\Si_2$ satisfies hypothesis of Theorem~\ref{teo:gcdmult1}, as $m_{\Si_2}=1$. Then $X(\Si_2)$ is pure.
\end{example}

The following is a counterexample showing that a converse of
 Theorem \ref{teo:gcdmult1} cannot hold.

\begin{example} \label{es:puro}
Let $\widehat{V}$ be the fan matrix of Example \ref{es:impuro}.
Consider the matrix
 $$\beta'=\begin{pmatrix}  1 &0& 0\\ 0& 2& 0 \\ 0& 0& 1\end{pmatrix} $$

and put
$$V':=\beta'\cdot\widehat{V}=\begin{pmatrix}  1 &-1& 2& -3& -1\\ 2& -2& -2& 4& -2\\ 1& 1& 1& 1& -5\end{pmatrix}.$$
A torsion matrix $\Gamma'$ with entries in $\Z/2\Z$ such that $Q\oplus\Gamma'$ represents the morphism assigning to each divisor its class is given by
\begin{equation}
\label{eq:Gamma}
\Gamma'=\left(
           \begin{array}{ccccc}
             0 & 0 & 0 & 1 & 1 \\
           \end{array}
         \right)
\,.\end{equation}
Let $\Sigma'_1$ be the fan in $\SF(V)$ corresponding to $\widehat{\Sigma}_1$ and $X'=X(\Sigma'_1)$. In this case $X'$ is a pure variety.
In fact, a basis of $\Cart(X')$ is given by the rows of the following matrix
$$C_{X'}=\begin{pmatrix} 40 &0 &0&0&0\\
0 &60&0&0&0\\
-20 &-30 &3&0&0\\
-8 &-48&0&2&0\\
15 &37 &-2&-1&1
\end{pmatrix}.$$
Then
$$
Q\cdot C_{X'}^T =\begin{pmatrix}
120 &60&-60&-60&60\\
120 &120&-120 &-120&120
\end{pmatrix},\quad\quad
\Gamma'\cdot C_{X'}^T = \begin{pmatrix}
0&0&0&0&0
\end{pmatrix}
$$
Then we see that $\Pic(X')$ is generated in $\Cl(X')\cong  \Z^2\oplus\Z/2\Z$  by the elements
$$ (120,120), (60,120)$$
 so that $X'$ is pure. On the other hand $m_{\widehat{\Si}_1}=2=\det(\b')$, so proving that a converse of Theorem~\ref{teo:gcdmult1} cannot hold.
\end{example}

\subsection{The case $|\SF(V)|=1$}\label{ssez:SF1}
The aim of this section is to exhibit a large class of pure toric varieties, by establishing the purity of every $\Q$-factorial complete toric variety $X=X(\Sigma)$ whose fan matrix $V$ admits a unique simplicial and complete fan given by $\Sigma$ itself. Geometrically, this property means that a small $\Q$-factorial modification of $X$ is necessarily an isomorphism, as explained in the Introduction.

We need a few preliminary lemmas. If $V$ is an $F$-matrix we put
\begin{eqnarray*}
\mathcal{I}_{V,tot}&=&\{I\subseteq\{1,\ldots, n+r\}\ |\ |I|=r \hbox{ and } \det(V^I)\not=0 \}\\
\mathcal{I}_{V,min}&=&\{I\in\mathcal{I}_{V,tot}\ |\ \langle V^I\rangle \hbox{ does not contain any column of $V$}\\
&&\hskip4cm \text{apart from its generators }\}.
\end{eqnarray*}

 \begin{lemma}\label{lem:mminmtot} Put
 \begin{eqnarray*}
m_{V,tot}&=&\gcd\{\det(V^I)\ |\  I\in\mathcal{I}_{V,tot}\}\\ m_{V,min}&=&\gcd\{\det(V^I)\ |\  I\in\mathcal{I}_{V,min}\}; \end{eqnarray*}
 then $m_{V,min}=m_{V,tot}$.
 \end{lemma}
 \begin{proof} Since $\mathcal{I}_{V,min}\subseteq \mathcal{I}_{V,tot}$ we have $m_{V,tot}|m_{V,min}$  We firstly show that the assertion is true when $m_{V,tot}=1$.  Otherwise there would exist a prime number $p$ dividing $\det(V^I)$ for every $I\in\mathcal{I}_{V,min}$; and there would exist $I_0\in\mathcal{I}_{tot}$ such that $p\not | \det(I_0)$. We choose such an $I_0$ with the property that the number $n_0$ of columns of $V$ belonging to $\langle V^{I_0}\rangle $ is minimum. Let $\v_1,\ldots, \v_n$ be the columns of $V^{I_0}$ and let $\v^*\in\langle V^{I_0}\rangle $ be a column of $V$ different from $\v_i$ for every $i$; then we can write $\v^*=\sum_{i=1}^n \frac{a_i} b \v_i$ with $a_i,b\in\Z$ and $(a_1,\ldots, a_n,b)=1$. For $i=1,\ldots, n$ let $\sigma_i=\langle \v_1,\ldots, \v_{i-1},\v^*,\v_{i+1},\ldots, \v_n\rangle $. Then $|\det(\sigma_i) |=\left(\frac{a_i} b\right)^n |\det(V^{I_0})|$ and $p$ divides $\det(\sigma_i)$ by the minimality  hypothesis on $I_0$. It follows that $p$ divides $a_i$ for $i=1,\ldots, n$; therefore $\v^*\in p\Z^n$ and this is a contradiction because $V$ is a fan matrix, hence reduced (see \cite[Def. 3.13]{RT-LA&GD}).\\
 Suppose now that $m_{V,tot}\not=1$. Then by \cite[Prop. 3.1 (3)]{RT-LA&GD} there exist a $CF$-matrix $\widehat V$ such that $V=\beta\widehat{V}$ for some $\beta\in \M_n(\Z)\cap\GL_n(\Q)$; and $m_{\widehat{V},tot}=1$ by \cite[Prop. 2.6]{RT-LA&GD}, so that we can apply the first part of the proof to $\widehat{V}$ and deduce that $m_{\widehat{V},min}=1$. Notice that $\mathcal{I}_{V,tot}=\mathcal{I}_{\widehat{V},tot}$, $\mathcal{I}_{V,min}=\mathcal{I}_{\widehat{V},min}$ and $\det(V^I)=\det(\beta)\det(\widehat{V}^I)$ for every $I\in\mathcal{I}_{V,tot}$, so that $m_{V,min}=\det(\beta)m_{\widehat{V},min}$ and $m_{V,tot}=\det(\beta)m_{\widehat{V},tot}$. It follows that $m_{V,min}=m_{V,tot}=\det(\beta)$.
 \end{proof}

\begin{lemma}\label{lem:esistenzafan}  Let $\Sigma_0$ be a simplicial fan in $\mathbb{R}^n$ such that $\sigma=|\Sigma_0|$ is a full dimensional convex cone.   Let $\w_1,\ldots, \w_k\in \mathbb{R}^n$ be  such that $\w_i\not\in \sigma$ for $i=1,\ldots, k$. There exists a simplicial fan $\Sigma$ in $\R^n$ such that
\begin{itemize}
\item[a)] $|\Sigma|= \sigma+\langle \w_1,\ldots, \w_k\rangle$;
\item[b)] $\Sigma(1)= \Sigma_0(1)\cup \{\langle\w_1\rangle ,\ldots, \langle \w_k\rangle \}$;
\item[c)] $\Sigma_0 \subseteq \Sigma$.

\end{itemize}   \end{lemma}
\begin{proof}
By induction on $k$. For the case $k=0$, we take $\Sigma=\Sigma_0$.  Assume that the result holds true for $k-1$. Let $\mathcal{W'}= \sigma +\langle \w_1,\ldots, \w_{k-1}\rangle$, $\mathcal{W}=\sigma+\langle \w_1,\ldots, \w_{k}\rangle$. By inductive hypothesis there exists a simplicial fan $\Sigma'$  such that  $|\Sigma'|=\mathcal{W'}$, $\Sigma'(1)= \Sigma_0(1)\cup \{ \langle  \w_1\rangle ,\ldots, \langle \w_{k-1}\rangle \}$ and
$\Sigma_0\subseteq \Sigma'$. We distinguish two cases:
\begin{itemize}
\item {\bf Case 1:}  $\w_{k}\in \mathcal{W'}$, so that $\mathcal{W}=\mathcal{W}'$; let $\tau$ be the minimal cone  in $\Sigma'$ containing $\w_{k}$. We take $\Sigma=s(\w_k,\tau)\Sigma'$, the stellar subdivision of $\Sigma'$ in direction $\w_k$ (see \cite[Def.\,III.2.1]{Ewald96}). Concretely, every $m$-dimensional cone $\mu=\langle \x_1,\ldots, \x_m\rangle\in \Sigma$ containing $\w_k$ is replaced by the set of the $m$-dimensional cones of the form $\langle \x_1,\ldots, \x_{i-1}, \w_k,\x_{i+1},\ldots, \x_m\rangle$. Conditions $a)$ and $b)$ are immediately verified. For condition $c)$ notice that, since $\w_{k}\not \in \sigma$, $\tau$  is not a face of any cone in $\Sigma_0$; therefore $\Sigma_0\subseteq \Sigma$.
\item {\bf Case 2:} $\w_{k}\not\in  \mathcal{W'}$, so that $\mathcal{W'}\subsetneq \mathcal{W}$. Let $\mathcal{F}$ be the set of facets  $f$ in $\Sigma'(n-1)$ which are cutted out by an hyperplane stricly separating $\mathcal{W}'$ and $\w_k$; that is $f\subseteq \partial\mathcal{W}'$, $f\not\subseteq\partial\mathcal{W}$ and   the cone $\tau_f=\langle f, \w_{k}\rangle$ is $n$-dimensional.\\
 Notice that $\mathcal{F}\not= \emptyset$: in fact
 $\mathcal{W}'$ is a convex polyhedral cone and  $\w_k\not\in \mathcal{W}'$; then there is an hyperplane $H$ cutting a facet $\varphi$ of  $\mathcal{W}'$ and strictly separating $\mathcal{W}'$ and $\w_k$; let $f$ be a facet of $\Sigma'$ contained in $\varphi$; then $f\in \mathcal{F}$.

\end{itemize}   Consider the set of simplicial cones
$$\Sigma=\Sigma'\cup\{\tau\ |\ \tau \preceq\tau_{f}\ \text{for some}\ f\in\mathcal{F} \}.$$
We claim that $\Sigma$ is a fan. By construction it is closed by faces, so that it suffices to show that $\tau_1\cap\tau_2$ is a face of both $\tau_1$ and $\tau_2$,  whenever $\tau_1,\tau_2\in\Sigma$.  Let $\tau_1,\tau_2\in\Sigma$. If they are both in $\Sigma'$ then $\tau_1\cap\tau_2$ is a face of $\tau_1,\tau_2$ because $\Sigma'$ is a fan. Assume that $\tau_1\in\Sigma'$ and $\tau_2\not\in \Sigma'$; then $\tau_2=\langle \tau, \w_k\rangle$, where $\tau$ is a face of some $f\in\mathcal{F}$. Let $H$ be the hyperplane cutting $f$; then $\w_k$ lies on the other side of $H$ with respect to $\mathcal{W'}$, so that $\tau_1\cap\tau_2\subseteq f$. Therefore $\tau_1\cap \tau_2=\tau_1\cap \tau \in\Sigma'$, so that it is a face of both $\tau_1$ and $\tau$ by induction hypothesis; but $\tau\preceq \tau_1$ so that $\tau_1\cap \tau_2\preceq \tau_2$. Finally, assume that both  $\tau_1$ and $\tau_2$ are not in $\Sigma'$. This means that there are facets $f_1,f_2$ in $\mathcal{F}$ and faces $\mu_1\preceq f_1,\mu_2\preceq f_2$ such that $\tau_1=\langle \mu_1,\w_k\rangle $ and $\tau_2=\langle \mu_2,\w_k\rangle $. We show that $\tau_1\cap \tau_2=\langle \mu_1\cap \mu_2,\w_k \rangle\in \Sigma$. Let $\x\in \tau_1\cap \tau_2$: then we can write $\x=\mathbf{y}_1+\lambda_1\w_k=\mathbf{y}_2+\lambda_2\w_k$, where $\mathbf{y}_1\in\mu_1, \mathbf{y}_2\in\mu_2$ and $\lambda_1, \lambda_2 \geq 0$. Whitout loss of generality we can assume $\lambda_1\geq\lambda_2$; put $\lambda =\lambda_1-\lambda_2$; then $\mathbf{y}_1+\lambda\w_k=\mathbf{y}_2$. Let $H$ be the hyperplane cutting $f_1$; since $f_1\in\mathcal{F}$, $\w_k\not\in H$, so that there esists a vector  $\n$ be a normal to $H$ such that $\n\cdot \x\leq 0$ for every $\x\in \mathcal{W}'$ and $\n\cdot \w_k>0$. Then $\n\cdot\ \mathbf{y}_2\leq 0$ and $\n\cdot (\mathbf{y}_1+\lambda\w_k)=\n\cdot\lambda\w_k\geq 0$, so that $\lambda =0$; this implies $\mathbf{y}_1=\mathbf{y}_2\in\mu_1\cap\mu_2$ and $\x\in\langle \mu_1\cap\mu_2,\w_k\rangle$. \\
Now we show that condition $a)$ holds for $\Sigma$.
By construction $|\Sigma|=|\Sigma'|\cup\bigcup_{f\in\mathcal{F}} \tau_f\subseteq \mathcal{W'}+\langle \w_k\rangle=\mathcal{W}$; conversely, let $\x\in\mathcal{W}$; if $\x\in\mathcal{W}'$ then $\x\in |\Sigma'|\subseteq |\Sigma|$; if $\x\not\in \mathcal{W}'$ then $\x=\mathbf{y}+\lambda\w_k$ for some $\mathbf{y}\in \mathcal{W}'$ and $\lambda>0$; up to replacing $\mathbf{y}$ by $\mathbf{y}+\mu\w_k$ for some $0\leq \mu <\lambda$ we can assume that $\mathbf{y}+\epsilon\w_k\not\in \mathcal{W}'$ if $\epsilon >0$.   Then for every $\epsilon$ there exists an hyperplane $H_\epsilon$ cutting a facet $\varphi_\epsilon$ of $\mathcal{W}'$ which separates $\mathcal{W}'$ and $\mathbf{y}+\epsilon\w_k$; since the facets of $\mathcal{W}'$ are finitely many,  by the pigeonhole principle  $H_\epsilon,\varphi_\epsilon$ do not depend on $\epsilon$ for $\epsilon \to 0$; call them $H,\varphi$ respectively. Let $\mathbf{n}$ be a normal vector to  $H$ such that $\mathbf{n}\cdot \mathbf{y}\leq 0$ and $\mathbf{n}\cdot (\mathbf{y}+\epsilon \w_k)>0$ for $\epsilon\to 0$; the existence of such  $\mathbf{n}$  implies $\mathbf{n}\cdot \mathbf{y}=0$ and $\mathbf{n}\cdot \w_k>0$, so that $\mathbf{y}\in\varphi$ and $\w_k\not\in H$. Then there is a facet $f\in \mathcal{F}$ such that $\mathbf{y}\in f$; therefore $\x\in \tau_f$, and $a)$ is proved. We showed that $\Sigma\setminus\Sigma'\not=\emptyset$; and every cone in $\Sigma \setminus \Sigma'$ has $\w_k$ as a vertex and all other vertices in $\Sigma'(1)$. Then condition   $b)$  is verified. Condition $c)$ is obvious since $\Sigma_0\subseteq \Sigma'\subseteq \Sigma$.
\end{proof}
\begin{corollary}\label{cor:esistenzauno} Let $V$ be a fan matrix. Then for every $I\in\mathcal{I}_{V,min}$ the cone $\langle V^I\rangle$ belongs to a fan in $\SF(V)$.
\end{corollary}
\begin{proof} It suffices to apply Lemma \ref{lem:esistenzafan} in the case $\Sigma_0=\{\tau\ |\ \tau\preceq \langle V^I\rangle \}$ and  $\w_1,\ldots, \w_k$ are the columns of $V_I$.
\end{proof}
Corollary \ref{cor:esistenzauno} has  the following immediate consequence:
\begin{corollary}\label{cor:esistenzadue} Let $V$ be a fan matrix such that $\SF(V)$ contains a unique fan $\Sigma$. Then for every $I\in\mathcal{I}_{V,min}$ the cone $\langle V^I\rangle$ belongs to $\Sigma$.
\end{corollary}

We are now in position to prove our purity condition:

\begin{proposition}\label{prop:unsolofan}
Let $X$ be a $\Q$-factorial complete toric variety and let $V$ be a fan matrix of $X$. Assume that $\SF(V)$ contains a unique fan.  Then $X$ is pure.
\end{proposition}
\begin{proof}  Let $Y=Y(\widehat{\Sigma})$ be  the universal $1$-covering of $X$ and let $\widehat{V}$ be a fan matrix associated to $Y$. Then  $\widehat{V}$ is a $CF$-matrix, so that  $m_{\widehat{V},tot}=1$ by \cite[Prop. 2.6 and Def. 2.7]{RT-LA&GD}. By Corollary \ref{cor:esistenzadue},
$ \mathcal{I}^{\Si}= \mathcal{I}^{\widehat{\Si}}=\mathcal{I}_{\widehat{V},min}$ so that
  $m_{\widehat{\Sigma}}=m_{\widehat{V},min}$ and, by Lemma \ref{lem:mminmtot}, the latter is equal to $m_{\widehat{V},tot}=1$. Then $X$ is pure by Corollary \ref{cor:gcdmult}.
\end{proof}

\begin{remark}
 Proposition \ref{prop:unsolofan} implies that the following  toric varieties are pure:
 \begin{itemize}
  \item $2$-dimensional $\Q$-factorial complete toric varieties
  \item  toric varieties whose universal $1$-covering is a product of weighted projective spaces.
  \end{itemize}
\end{remark}

\begin{remark}
In the case $|\SF(V)|=1$ the unique complete and $\Q$-factorial toric variety $X$ whose fan matrix is $V$ is necessarily projective. This is a consequence of the fact that  $\mathrm{Nef}(X)=\mathrm{\overline{Mov}}(X)$, recalling that the latter is a full dimensional cone, by \cite[Thm. 2.2.2.6]{ADHL}.
\end{remark}

\subsubsection{An application to completions of fans} \label{sssez:completamento}
Lemma \ref{lem:esistenzafan} can be applied to give a complete refinement $\Si$ of a given fan $\Si_0$ satisfying the further additional hypothesis:
\begin{itemize}
  \item[(*)] \emph{assume that $|\Si|=\Si_0 + \langle \w_1,\ldots, \w_k\rangle=\R^n$\,.}
\end{itemize}
In particular, if we consider the fan $\Si'=\Si_0\cup\{\langle\w_1\rangle\,\ldots,\langle\w_k\rangle\}$ then Lemma~\ref{lem:esistenzafan} \emph{gives a completion $\Si$ of $\Si'$ without adding any new ray}.

The latter seems to us an original result. In fact, it is actually well known that every fan $\Si'$ can be refined to a complete fan $\Si$ (see \cite[Thm.~III.2.8]{Ewald96}, \cite{EwaldIshida06} and the more recent \cite{Rohrer11}). Anyway, in  general the known completion procedures need the addition of some new ray, so giving $\Si'(1)\subsetneq\Si(1)$. As observed in the Remark following the proof of \cite[Thm.~III.2.8]{Ewald96}, just for $n=3$ ``completion without additional 1-cones can be found'' but this fact does no more hold for $n\geq 4$: at this purpose, Ewald refers the reader to the Appendix to section III, where he is further referred to a number of references. Unfortunately we were not able to recover, from those references and, more generally, from the current literature, as far as we know, an explicit example of a 4-dimensional fan which cannot be completed without adding some new ray.
 For this reason, we believe that the following example may fill up a lack in the literature on these topics.

 \begin{example} \label{ex:noncompletabile}
Consider the fan matrix
\begin{equation*}
  V= \left(
  \begin{array}{ccccccc}
    1&0&0&0&0&-1&1 \\
    0&1&0&0&-1&-1&2 \\
    0&0&1&0&-1&0&1 \\
    0&0&0&1&-1&-1&1 \\
  \end{array}
\right)
\end{equation*}
and consider the fan $\Si$ given by taking all the faces of the following three maximal cones generated by columns of $V$
\begin{equation*}
  \Si(4)=\left\{\langle2,3,4,6\rangle,\langle2,4,5,7\rangle,\langle1,4,5,6\rangle\right\}
\end{equation*}
The fact that $\Si$ is a fan follows immediately by easily checking that
\begin{eqnarray*}
\langle2,3,4,6\rangle\cap\langle2,4,5,7\rangle&=&\langle2,4\rangle\\
\langle2,3,4,6\rangle\cap\langle1,4,5,6\rangle&=&\langle4,6\rangle\\
\langle1,4,5,6\rangle\cap\langle2,4,5,7\rangle&=&\langle4,5\rangle\,.
\end{eqnarray*}
Notice that $\Si$ is not a complete fan since e.g. the 3-dimensional cone
$$\langle2,3,6\rangle=\left\langle\begin{array}{ccc}
                                    0 & 0 & -1 \\
                                    1 & 0 & -1 \\
                                    0 & 1 & 0 \\
                                    0 & 0 & -1
                                  \end{array}
\right\rangle$$
is a facet of the unique cone
$$\langle2,3,4,6\rangle=\left\langle\begin{array}{cccc}
                                    0 & 0 & 0 & -1 \\
                                    1 & 0 & 0 & -1 \\
                                    0 & 1 & 0 & 0 \\
                                    0 & 0 & 1 & -1
                                  \end{array}\right\rangle\in\Si(4)\,.$$
Moreover it cannot be completed since every further maximal cone admitting $\langle2,3,6\rangle$ as a facet does not intersect correctly the remaining cones in $\Si(4)$. In fact
\begin{itemize}
  \item $\langle1,2,3,6\rangle\cap\langle2,4,5,7\rangle\supsetneq \langle2\rangle$\,: consider e.g. $$\v=\left(
                                                                                                     \begin{array}{cccc}
                                                                                                        1&1&0&0 \\
                                                                                                     \end{array}
                                                                                                   \right)^T\in
                                                                                                   \langle1,2\rangle\cap\langle5,7\rangle$$

  \item $\langle2,3,5,6\rangle\cap\langle1,4,5,6\rangle\supsetneq \langle5,6\rangle$\,: consider e.g.
  $$\w=\left(
        \begin{array}{cccc}
          0 & -2 & -1 & -2 \\
        \end{array}
      \right)^T\in\langle3,5\rangle\cap\langle1,5,6\rangle\quad\text{but}\quad\w\not\in\langle5,6\rangle
      $$
  \item $\langle2,3,6,7\rangle$ is not a maximal cone.
\end{itemize}
 \end{example}

\section{A characterization of $\Pic(X)$ for some pure toric variety}\label{sez:pic}

Let $X=X(\Sigma)$ be a complete $\Q$-factorial toric variety having $V$ as a fan matrix; let $Y$ be its universal $1$-covering, $\widehat V$ be a fan matrix associated to $Y$ and $V=\beta \widehat V$.
Recall that a Weil divisor $L=\sum_{j=1}^{n+r}a_jD_j$ is a Cartier divisor if it is locally principal, that is
$$\forall I\in\mathcal{I}^\Sigma\ \exists \mathbf{m}_I\in M \hbox{ such that } \mathbf{m}_I\cdot \mathbf{v}_j=a_j, \forall j\not\in I. $$
Let $\Cart(X)$ be the group of torus invariant Cartier divisors of $X$. Then
$$\Cart(X)=\bigcap_{I\in \mathcal{I}^\Sigma}\mathcal{L}_r(V^I)=\bigcap_{I\in \mathcal{I}^\Sigma}(\mathcal{L}_r(V)\oplus E_I) $$
recalling notation (\ref{EI}).
The Picard group $\Pic(X)$ of $X$ is the image of $\Cart(X)$ in $\Cl(X)$, via morphism $d_X$ (recall here and in the following, notation introduced in diagram (\ref{div-diagram-covering})).
\\
In \cite[Thm.~2.9.2]{RT-LA&GD} we showed that if $Y$ is a PWS then we can identify
\begin{equation}\label{eq:picY}
\Pic(Y)=\bigcap_{I\in \mathcal{I}^\Sigma} \mathcal{L}_c(Q_I)\subseteq \Z^r.\end{equation}
Let $\x\in\Pic(Y)$. For $I\in \mathcal{I}^\Sigma$ we can write $\x=Q\cdot \mathbf{a}_I$ where $\mathbf{a}_{I}\in E_I$\,. If $I,J\in \mathcal{I}^\Sigma$ put
\begin{equation}\label{uIJ}
  \u_{IJ}=\mathbf{a}_I-\mathbf{a}_J\in\ker(Q)=\mathcal{L}_r(\widehat V)\,.
\end{equation}
Let $\z\in\Cart(Y)$ such that $Q\cdot\z=\x$. By definition, for every $I\in \mathcal{I}^\Sigma$ there is a unique decomposition $\z=\t(I)+\mathbf{a}_I$ with $\t(I)\in\mathcal{L}_r(\widehat{V})$. Moreover, \begin{equation}
\label{eq:cartX}
\z\in\Cart(X)\Leftrightarrow \t(I)\in\mathcal{L}_r(V), \forall I\in\mathcal{I}^\Sigma.\end{equation}
\begin{proposition}\label{prop:piccharacterization} $\x\in\overline{\a}(\Pic(X))$ if and only if $\x\in\Pic(Y)$ and $\u_{IJ}\in\mathcal{L}_r(V)$, for every $I,J\in\mathcal{I}^\Sigma$, where $\u_{IJ}$ is defined by (\ref{uIJ}).
\end{proposition}
\begin{proof} Suppose that $\x\in\overline{\a}(\Pic(X))$. Then there exists $\z\in\Cart(X)$ such that $Q\cdot \z=\x$.  For every $I\in \mathcal{I}^\Sigma$ consider the decomposition $\z=\t(I)+\mathbf{a}_I$ with $\t(I)\in\mathcal{L}_r(V)$. Then $\u_{IJ}=\mathbf{a}_I-\mathbf{a}_J= \t(J)-\t(I)\in\mathcal{L}_r(V)$ for every $I,J\in I^\Sigma$. Conversely, suppose that $\u_ {IJ}\in\mathcal{L}_r(V)$ for every $I,J\in\mathcal{I}^\Sigma$. Let $\z'\in\Cart(Y)$ be such that $Q\cdot \z'=\x$. For every $I\in \mathcal{I}^\Sigma$ there is a decomposition $\z'=\t'(I)+\mathbf{a}_I$ with $\t'(I)\in\mathcal{L}_r(\widehat V)$. Fix $I_0\in \mathcal{I}^\Sigma$ and put $\z=\z'-\t'(I_0)$. We claim that $\z\in\Cart(X)$. Indeed, let $I\in\mathcal{I}^\Sigma$ and decompose $\z=\t(I)+\mathbf{a}_I$ with $\t(I)\in\mathcal{L}_r(\widehat V)$ and  $\t(I_0)=\mathbf{0}$. It follows that for every $I\in\mathcal{I}^\Sigma$
$$\t(I)=\t(I)-\t(I_0)=\mathbf{a}_{I_0}-\mathbf{a}_I=\u_{I_0I}\in\mathcal{L}_r(V).$$
\end{proof}

\begin{theorem}\label{thm:pic}
Let $X$ be a pure $\Q$-factorial complete toric variety and choose an isomorphism $\Cl(X)\cong \Z^r\oplus T$ such that $\Pic(X)$ is mapped in $\Z^r$. Then the following characterization of $\Pic(X)$ holds:
\begin{equation*}
  \x\in\Pic(X)\,\Leftrightarrow\,\forall\,I,J\in\mathcal{I}^\Sigma\quad \x\in \bigcap_{I\in \mathcal{I}^\Sigma} \mathcal{L}_c(Q_I)\ \text{and}\ \u_{IJ}\in\mathcal{L}_r(V)
\end{equation*}
where $\u_{IJ}$ is defined by (\ref{uIJ}).
\end{theorem}
\begin{proof}
Define $s:\Z^r\to \Z^r\oplus T$ by $s(a)=(a,0)$. Then $\overline{\alpha}\circ s=id_{\Z_r}$ and $s\circ\overline{\alpha}|_{\Pic(X)}=id_{\Pic(X)}$. Then we have
for every $\x\in\Z^r$
$$\x\in\overline{\alpha}(\Pic(X))\Leftrightarrow s(\x)\in\Pic(X).$$
The result follows from Proposition \ref{prop:piccharacterization} by identifying $\x$ and $s(\x)$.
\end{proof}

\begin{example}
Let $\Sigma'_1$ be the fan defined in Example \ref{es:puro}. Then
$$\mathcal{I}^{\Sigma'_1}=\mathcal{I}^{\widehat{\Sigma}_1}=\{\{1,3\},\{2,3\}, \{3,5\},\{1,4\},    \{2,4\},\{4,5\}\}.$$
so that $$ \bigcap_{I\in \mathcal{I}^{\Sigma'_1}} \mathcal{L}_c(Q_I)=\Z (30,0)\oplus \Z (0,60).$$
Let $\x=(30x, 60y)\in \bigcap_{I\in \mathcal{I}^{\Sigma'_1}} \mathcal{L}_c(Q_I)$, with $x,y\in \Z$. For $I\in \mathcal{I}^{\Sigma'_1}$ we can write $\x=Q\cdot \mathbf{a}_I$ where $\mathbf{a}_{I}\in E_I$. The $\mathbf{a}_I$'s are easily calculated:
$$\begin{array}{lllcccccl}
\mathbf{a}_{\{1,3\}}&=&(&20y &0 &3x-6y & 0 & 0&)\\
\mathbf{a}_{\{2,3\}}&=& (&0 &30y &3x-3y & 0 & 0& )\\
\mathbf{a}_{\{3,5\}}&=& (&0 &0 &3x-24y & 0 & 60y& )\\
\mathbf{a}_{\{1,4\}}&=& (&20y &0 &0 & 5x-10y & 0& )\\
\mathbf{a}_{\{2,4\}}&=& (&0 &30y &0 & 5x-5y & 0& )\\
\mathbf{a}_{\{4,5\}}&=& (&0 &0 &0 & 5x-40y & 60y& )
\end{array}$$

Notice that, with the notation of (\ref{uIJ}),  $\u_{IJ}=\u_{IK}-\u_{KJ}$; then in order to calculate $\u_{IJ}$ for every $I,J\in\mathcal{I}^{\Sigma'_1} $ it suffices to compute $\u_{I_jI_{j+1}}$ for a sequence $I_1,\ldots , I_s$ such that $\langle Q_{I_j}\rangle $ and $\langle Q_{I_{j+1}}\rangle $ have a common facet and $\mathcal{I}^{\Sigma'_1}=\{I_1,\ldots , I_s\}$; in this way we obtain vectors having at most $r+1=3$ nonzero components:
$$\begin{array}{lllcccccl}
\mathbf{u}_{\{1,3\}\{2,3\}}&=& (& 20y&-30y&-3y&0&0& )\\
\mathbf{u}_{\{2,3\}\{3,5\}} & =& (& 0& 30y&21y&0&-60y&)\\
\mathbf{u}_{\{3,5\}\{4,5\}} &=& (&  0&0&3x-24y&-5x+40y& 0& )\\	
\mathbf{u}_{\{4,5\}\{1,4\}}&=& (& -20y&0&0&-30y&60y&)\\
\mathbf{u}_{\{1,4\}\{2,4\}}&=& (& 20y&-30y&0&-5y&0 &) \end{array}	
$$
Multiplying by the matrix $\Ga'$ found in (\ref{eq:Gamma}) we obtain
$$
\Ga'\cdot  \mathbf{u}_{\{1,3\}\{2,3\}}^T=0;\quad   \Ga'\cdot  \mathbf{u}_{\{2,3\}\{3,5\}}^T=-60y;\quad \Ga'\cdot  \mathbf{u}_{\{3,5\}\{4,5\}}^T=-5x+40y;$$
$$\Ga'\cdot  \mathbf{u}_{\{4,5\}\{1,4\}}^T=30y;\quad \quad\Ga'\cdot  \mathbf{u}_{\{1,4\}\{2,4\}}^T=-5y.
$$
Recall that $\Gamma'$ takes values in $\Z/2\Z$ and that for every $\u\in\mathcal{L}_r(\widehat V)$
$$\Ga'\cdot \u^T=0\hbox{ if and only if } \u\in\mathcal{L}_r(V');$$ then we see that $\u_{IJ}\in\ker(\Ga')$ for every $I,J\in \mathcal{I}^{\Sigma'_1}$ if and only if $x,y\in 2\Z$, that is if and only if $\x\in\Z (60,0)\oplus \Z(0,120)$. By Theorem \ref{thm:pic}, $\Pic(X')$ can be identified with the subgroup $\Z (60,0)\oplus \Z(0,120)$  in $\Cl(X')\simeq \Z^2\oplus \Z/2\Z$, according to what we established in Example \ref{es:puro}.
\end{example}

\bibliography{MILEA}
\bibliographystyle{acm}

\end{document}